\documentclass[11pt]{article}
= 0.0in \headheight = 0.0in \topmargin = 0.2in
\def\Dj{\hbox{D\kern-.73em\raise.30ex\hbox{-}
\raise-.30ex\hbox{}}}
\def\dj{\hbox{d\kern-.33em\raise.80ex\hbox{-}
\raise-.80ex\hbox{\kern-.40em}}}

\usepackage{amssymb}
\usepackage{color}
\usepackage{amssymb}
\usepackage{float}
\usepackage{amsmath}
\usepackage{caption}
\usepackage{graphicx}
\usepackage{latexsym}
\usepackage{amsthm}
\usepackage{tikz}
\usepackage{lineno}
\usepackage{enumitem}
\usepackage{epsfig}
\usepackage{graphicx}
\usepackage{amsmath,amsthm,amsfonts,amssymb,amscd,cite}
\allowdisplaybreaks
\usepackage{mathrsfs}

\newtheorem{question}{Question}[section]
\newtheorem{definition}{Definition}[section]
\newtheorem{theorem}{Theorem}[section]
\newtheorem{lemma}{Lemma}[section]

\newtheorem{corollary}{Corollary}[section]
\newtheorem{proposition}{Proposition}[section]

\newtheorem{remark}[theorem]{Remark}

\begin{document}

%

%
%

\title{The number of dissociation sets in  connected graphs}
\author{Bo-Jun Yuan\thanks{ Supported by National Natural Science Foundation of
China(12201559).}, Ni Yang, Hong-Yan Ge and Shi-Cai Gong\thanks{Corresponding author.}
\thanks{ Supported by National Natural Science Foundation of
China(12271484).}~\thanks{ E-mail addresses: ybjmath@163.com(B.J. Yuan),  yangni2024@163.com(N. Yang), and
scgong@zafu.edu.cn(S.C. Gong).}
\\{\small \it  School of Science, Zhejiang University of Science and Technology, }\\{\small \it
Hangzhou, 310023, P. R. China}}

\date{}
\maketitle

\baselineskip=0.20in

\noindent {\bf Abstract.}
Extremal problems related to the enumeration of graph substructures, such as independent sets, matchings, and induced matchings, have become a prominent area of research with the advancement of graph theory. A subset of vertices is called a dissociation set if it induces a subgraph with vertex degree  at most $1$, making it a natural generalization of these previously studied substructures.

In this paper, we present efficient tools to strictly increase the number of dissociation sets in a connected graph. Furthermore, we establish that the maximum number of dissociation sets among all connected graphs of order $n$ is given by
\begin{align*}
\begin{cases}
2^{n-1}+(n+3)\cdot 2^{\frac{n-5}{2}}, &~ {\rm if}~  n~{\rm is}~{\rm odd};\\
2^{n-1}+(n+6)\cdot 2^{\frac{n-6}{2}}, &~ {\rm if}~  n~{\rm is}~{\rm even}.
\end{cases}
\end{align*}
Additionally, we determine the achievable upper bound on the number of dissociation sets in a tree of order $n$ and characterize the corresponding extremal graphs as an intermediate result.
Finally, we identify the unicyclic graph that is the candidate for having the second largest number of dissociation sets among all connected graphs.
\vspace{3mm}



\baselineskip=0.20in

\section{Introduction}
All graphs considered in this paper are   finite, simple and undirected.
We follow the terminology
in \cite{Bondy}. Let $G=(V,E)$ be a simple graph.
A subset of vertices in $G$ is called an {\it independent set} ({\it dissociation set})
if it induces a subgraph with vertex degree $0$ (at most $1$).
An independent set (dissociation set) of $G$ is {\it maximal}  or {\it maximum} if it is not a proper subset of any independent sets  (dissociation sets) or it has
maximum cardinality, respectively.
A matching $M\subseteq E$ is
called {\it induced} if the endpoints of $M$ induce a $1$-regular subgraph in $G$.
Similarly, an induced matching is called {\it maximal} or {\it maximum} if it is not a proper subset of any induced matchings or it has
maximum cardinality, respectively.
Clearly, a dissociation set encompasses both extreme cases: all vertices with degree $0$ and all vertices with degree $1$, serving as natural generalizations of independent sets and induced matchings.

Extremal problems concerning the enumeration of discrete substructures have a long history in  combinatorics and graph theory.
One main line is to study the enumeration of extreme graph substructures, such as maximum, maximal, minimum
and minimal.
For example, the earliest work can be traced back to 1960s, Erd\H{o}s and Moser first raised the problem of determining the maximum value of
the number of maximal independent sets in general graphs with fixed order and
those graphs achieving the maximum value.
Since then, a lot of work focuses on the maximum (and minimum) number of
maximum (and maximal) independent sets \cite{Griggs, Moon, Wilf}  (see a survey in \cite{Jou}), dissociation sets  \cite{Sun, Tu, Tu1, Zhang}, induced matchings  \cite{Basavaraju, 4, Gupta, Xiao, Yuan} for various classes of graphs.

A complementary main line involves the  enumeration of the total number of graph substructures.
Andrew Granville initially raised the problem of
determining the maximum value of independent sets in
 the family of $d$-regular graphs
of the same size in connection
to combinatorial number theory.
Later, the problem appeared first in print in a paper by Alon \cite{Alon}; for a survey see \cite{Zhaoyf}.
As a generalization, Wang and Tu \cite{Wangzy} et al. considered the maximum number of dissociate sets
in the family of  regular bipartite graphs. Zhou et al. \cite{Zhou} enumerated dissociation sets in grid graphs using the state matrix recursion algorithm.
And Chen and Liu \cite{Chenyan} characterized some types of graphs with minimum and maximum number of induced matchings.

Note that $2^n$  serves as a trivial upper bound for the number of dissociation sets in general graphs of order $n$, with the extremal graph achieving this bound being the disjoint union of $K_1$ and $K_2$.
The main purpose of our study is to confine the aforementioned problems to the domain of connected graphs.
In Section \ref{Preliminaries}, we  provide some efficient tools to make the number of dissociation sets of a connected  graph strictly increasing.
Utilizing these tools, we establish that the extremal enumeration problems for connected graphs can be effectively reduced to those of trees.
In Section \ref{tree}, we determine the maximum number of dissociation sets for trees of order $n$.
Subsequently,  we extend our results to connected graphs.
As an application, in Section \ref{unicyclic}  we characterize an achievable upper bound
on the  number of dissociation sets of a  unicyclic graph  with  the responding extremal graphs.

\section{Preliminaries\label{Preliminaries}}
Let $G$ denote a disconnected graph whose all components are complete graphs; and
for a complete graph  $K_r$ disjoint from $G$, let $K_r\ast G$ be a connected graph obtained by picking a vertex $u\in V(K_r)$ and connecting it to exactly one vertex of each component in $G$, see Figure \ref{fig1} for an example.

\begin{figure}[ht!]
\begin{center}
\begin{tikzpicture}[scale=0.9,style=thick]
\tikzstyle{every node}=[draw=none,fill=none]
\def\vr{3pt} 

\begin{scope}[yshift = 0cm, xshift = 0cm]
\path (0,0) coordinate (x1);
\path (-1,-1.5) coordinate (x2);
\path (-2.85,-1.5) coordinate (x3);
\path (2.65,-1.5) coordinate (x4);

\draw (x1) --(x2);
\draw (x1) --(x3);
\draw (x1) --(x4);
\draw(0,0.6) circle[radius=0.6];
\draw(-1,-2.1) circle[radius=0.6];
\draw(-3,-2.1) circle[radius=0.6];
\draw(2.8,-2.1) circle[radius=0.6];

\draw (x1)  [fill=black] circle ;
\draw (x2)  [fill=white] circle ;
\draw (x3)  [fill=white] circle ;
\draw (x4)  [fill=white] circle ;
\draw[above] (x1)++(0,0.3) node {$K_r$};
\draw[above] (x2)++(0,-0.9) node {$K_{s_2}$};
\draw[above] (x3)++(-0.15,-0.9) node {$K_{s_1}$};
\draw[above] (x4)++(0.15,-0.9) node {$K_{s_t}$};
\draw[above] (x2)++(1.95,-0.9) node {$\cdots$};

\end{scope}
\end{tikzpicture}
\end{center}
\caption{ $K_r\ast (K_{s_1}\cup K_{s_2}\cup \cdots \cup K_{s_t})$.}
\label{fig1}
\end{figure}
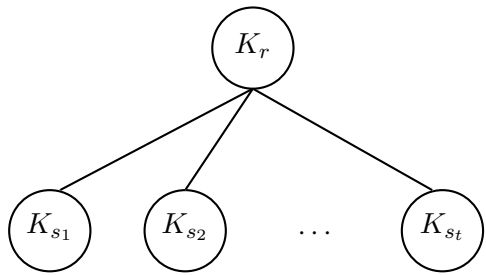

Let $G=(V(G),E(G))$ be a graph. A vertex of degree $1$ is called a   {\it pendant vertex} of $G$, and a  vertex of degree at least $2$ adjacent to a pendant vertex will be called  {\it quasi-pendant}.
Let $c(G)=|E(G)|-|V(G)|+1$ denote the dimension of the cycle space of a connected graph $G$.
For $S\subseteq V(G)$,  denote by $G[S]$ the subgraph of $G$  induced by $S$, and denote by $G-S$ the induced subgraph  $G[V(G)\setminus S]$.  In particular, if $S=\{v\},$ then we write $G-v$ instead of
$G-\{v\}$.
For a vertex $v \in V(G)$, the closed neighborhood and open neighborhood of $v$ are denoted by $N_G[v]$ and $N_G(v)$, respectively.
For a subset $S\subset  V(G)$, the closed neighborhood and open neighborhood of $S$ are defined as $N_G[S]=\cup_{v\in S}N_G[v]$ and $N_G(S)=N_G[S]\setminus S$, respectively.
Let $\mathcal{D}(G)$ be the set of dissociation sets of  $G$.

For a graph  $G$   with $v\in V(G)$, it is easy to see that $\mathcal{D}(G)$ can be partitioned into the following three parts:
\begin{equation*}
\begin{aligned}
D(G,\bar{v})&=\{~D~|~D\in \mathcal{D}(G)~ and~ v\notin  D\};\\
D(G, v^0)&=\{~D~|~D\in \mathcal{D}(G), v\in D~ and~ d_{G[D]}(v)=0\};\\
D(G, v^1)&=\{~D~|~D\in \mathcal{D}(G), v\in D~ and~ d_{G[D]}(v)=1\}.
\end{aligned}
\end{equation*}
Let $d(G,k)$ denote the number of dissociation sets consisting of $k$ vertices. Set $d(G,0)=1$
for all graphs. The total number of dissociation sets of $G$, denote by
$d(G)$, is defined as
$$d(G)=\sum_{k\geq 0}d(G,k).$$
Let $G_1,G_2,\ldots,G_s$  be all the components of $G$. It is easy to see that $d(G)=\prod_{i=1}^{i=s}d(G_i).$

Based on the definition of the total number of dissociation sets in a graph $G$,
 we can directly deduce both the trivial upper and lower bounds for $d(G)$ without the need for additional tools.
\begin{theorem} \label{y4}
Let $G$ be a graph of order $n$. Then $\frac{n^2+n+2}{2}\leq d(G)\leq 2^n$, the left equality holds if and only if $G\cong K_{n_1, n_2, \ldots, n_k}$ where $n_i\leq 2$ for any $1\leq i\leq k$ and $n_1+n_2+\cdots n_k=n$, and the right equality holds if and only if $G\cong sK_{1}\cup tK_{2}$ where $s+2t=n.$
\end{theorem}
\begin{proof}
For any graph $G$ of order $n$, $d(G,0)=1$, $d(G,1)= {n\choose 1}$ and $d(G,2)= {n\choose 2}$. So $d(G)\geq d(G,0)+d(G,1)+d(G,2)=\frac{n^2+n+2}{2}$ and the equality holds iff $d(G,3)=0$ iff $G$ contains neither $3K_1$ nor $K_2\cup K_1$ as induced subgraphs iff $G\cong K_{n_1, n_2, \ldots, n_k}$ where $n_i\leq 2$.
Note that $d(G)=d(G,0)+d(G,1)+\cdots +d(G,n)\leq {n\choose 0}+{n\choose 1}+\cdots +{n\choose n}=2^n$.
The equality holds implying that $d(G,3)={n\choose 3}$. Then the  subgraph  induced by any three vertices of graph $G$ is isomorphic to one of $3K_1$ or $K_2\cup K_1$. The result follows.
\end{proof}
In the proof of Theorem  \ref{y4},  we use the fact that if $d(G,3)=0$, then $d(G,k)=0$ for any $k\geq 4$.

Note that the upper bound in Theorem  \ref{y4} is attained in a disconnected graph. Naturally, we could consider the same problem in the set of connected graphs of order $n$.
The upper bound in Theorem \ref{y4}  plays a crucial role in our discussion.
For convenience, set
\begin{equation}\label{(2.1)}
\begin{aligned}
 g(n):=2^{n}.
\end{aligned}
\end{equation}

\begin{proposition} \label{y1}
Let $G$ be a graph with $v \in V(G)$. Then 
$$d(G)=d(G-v)+d(G-N[v])+\sum_{u\in N(v)}d(G-N[u]\cup N[v]).$$
\end{proposition}
\begin{proof}
Obviously,
\begin{equation*}
\begin{aligned}
d(G)& =|D(G,\bar{v})|+|D(G, v^0)|+|D(G, v^1)|\\
&=d(G-v)+d(G-N[v])+\sum_{u\in N(v)}d(G-N[u]\cup N[v]).
\end{aligned}
\end{equation*}
\end{proof}

From  Proposition  \ref{y1}, the number of dissociation sets of special graphs, such as a path $P_n$, a star $S_n$ and a cycle $C_n$, can be derived immediately.
\begin{proposition} \label{y2}
If $n\geq 3,$ then $d(P_n)=d(P_{n-1})+d(P_{n-2})+d(P_{n-3})$. $d(S_n)=n+2^{n-1}$.
If $n\geq 4,$ then $d(C_n)=d(P_{n-1})+d(P_{n-3})+2d(P_{n-4})$.
\end{proposition}

Intuitively, one would expect that the number of dissociation sets in a graph $G$
increases as the graph becomes more sparse. This hypothesis is substantiated by the following two results, which serve as the foundational tools for our Main Theorem.
\begin{proposition} \label{y2}
If $H$ is a proper induced  subgraph of a graph $G$, then $d(H)< d(G)$.
\end{proposition}
\begin{proof}
Let $D\in \mathcal{D}(H)$ be any dissociation set of $H$,  then $D$ is also a dissociation set of
$G$. Thus $d(H)\leq d(G)$. Denote by $S=G-H\neq \emptyset$.  Note that any  dissociation set of $S$ is also a dissociation set of $G$ but not contained in $H$.  This implies that  $d(H)< d(G)$.
\end{proof}

\begin{lemma} \label{01}
Let $G$ be a graph and let $uv\in E(G)$. Then 
$$d(G)\leq d(G-uv),$$ 
with equality if and only if $N_{G}[u]=N_{G}[v].$
\end{lemma}
\begin{proof}
Let $W_1=N_{G}(u)\setminus N_{G}[v]$, $W_2=N_{G}(u)\cap N_{G}(v)$, $W_3=N_{G}(v)\setminus N_{G}[u]$  and $W_4=V(G)\setminus (N_{G}[u]\cup N_{G}[v])$. For convenience, let  $G'=G-uv.$
By Proposition \ref{y1},
\begin{equation*}
\begin{aligned}
d(G)& =d(G-u)+\sum_{w\in N_{G}(u)}d(G-N_{G}[u]\cup N_{G}[w])+d(G-N_{G}[u])\\
&=d(G-u)+\sum_{w\in W_1\cup W_2}d(G-N_{G}[u]\cup N_{G}[w])+d(G-N_{G}[u]\cup N_{G}[v])+d(G-N_{G}[u])
\end{aligned}
\end{equation*}
and
\begin{equation*}
\begin{aligned}
d(G')& =d(G'-u)+\sum_{w\in N_{G'}(u)}d(G'-N_{G'}[u]\cup N_{G'}[w])+d(G'-N_{G'}[u])\\
&=d(G'-u)+\sum_{w\in W_1\cup W_2}d(G'-N_{G'}[u]\cup N_{G'}[w])+d(G'-N_{G'}[u]).
\end{aligned}
\end{equation*}
Note that  $G-u\cong G'-u,$ thus  $d(G-u)=d(G'-u).$

We claim that  $$\sum_{w\in  W_1\cup W_2}d(G-N_{G}[u]\cup N_{G}[w])\leq \sum_{w\in  W_1\cup W_2}d(G'-N_{G'}[u]\cup N_{G'}[w]).$$
In graph $G$, for any neighbor $w\in W_1$ we have the corresponding $w\in W_1$ in $G'$,
then $d(G-N_{G}[u]\cup N_{G}[w])=d(G[W_3\cup W_4 \setminus N_{W_3\cup W_4}(w)])$
and  $d(G'-N_{G'}[u]\cup N_{G'}[w])=d(G'[\{v\}\cup W_3\cup W_4 \setminus N_{W_3\cup W_4}(w)])$.
By Proposition \ref{y2}, $d(G[W_3\cup W_4 \setminus N_{W_3\cup W_4}(w)])< d(G'[\{v\}\cup W_3\cup W_4 \setminus N_{W_3\cup W_4}(w)])$. It follows that
$\sum_{w\in W_1}d(G-N_{G}[u]\cup N_{G}[w])<  \sum_{w\in W_1}d(G'-N_{G'}[u]\cup N_{G'}[w]).$
For the similar reason, it is easy to verify that
$\sum_{w\in W_2}d(G-N_{G}[u]\cup N_{G}[w])=\sum_{w\in W_2}d(G'-N_{G'}[u]\cup N_{G'}[w]).$
The claim follows. In particular, the equality in the claim  holds iff $W_1= \emptyset$.

We now claim that $$d(G-N_{G}[u]\cup N_{G}[v])+d(G-N_{G}[u])\leq d(G'-N_{G'}[u]).$$
Observe that $d(G-N_{G}[u])=d(G[W_3\cup W_4])$, $d(G-N_{G}[u]\cup N_{G}[v])=d(G[W_4])$ and
$d(G'-N_{G'}[u])=d(G'[\{v\}\cup W_3\cup W_4])$.
For the sake of simplicity, let $H=G'[\{v\}\cup W_3\cup W_4]$.
The claim is equivalent to verify that $d(G[W_3\cup W_4])+d(G[W_4])\leq d(H).$
Applying Proposition \ref{y1} to the vertex $v$ of the  subgraph $H$, we have
\begin{equation} \label{boi1}
\begin{aligned}
d(H)& =d(H-v)+d(H-N_{H}[v])+\sum_{w\in W_3}d(H-N_{H}[v]\cup N_{H}[w])\\
&=d(G[W_3\cup W_4])+d(G[W_4])+\sum_{w\in W_3}d(H-N_{H}[v]\cup N_{H}[w])\\
&\geq d(G[W_3\cup W_4])+d(G[W_4]).
\end{aligned}
\end{equation}
The  claim follows. Especially, if $W_3\neq \emptyset$, then $\sum_{w\in W_3}d(H-N_{H}[v]\cup N_{H}[w])> 0.$
It implies that the equality in $(1)$  holds iff $W_3= \emptyset$.

Combining the three upper bounds yields
$d(G)\leq d(G-uv)$ with equality if and only if $W_1=W_3= \emptyset$ (i.e., $N[u]=N[v]$).
%
\end{proof}

\begin{remark}
In addressing the problem of determining the extremal values of $d(G)$ within the set of connected graphs of a fixed order, Lemma \ref{01} suggests that we may confine our considerations to the subset of trees.
\end{remark}

\begin{remark}
As a direct application of Lemma \ref{01}, Theorem \ref{y4} is readily derived.
\end{remark}

Two distinct vertices $u,v\in V(G)$ are called {\it true twins}  if $N[u]=N[v]$ (which shows that $uv\in E(G)$), and are  called {\it false twins} if $N(u)=N(v)$. It is easy to see that if $u,v$ are true twins in graph $G$, then $u,v$ are false twins in the subgraph $G-uv.$
\begin{corollary} \label{coro1}
Let $C_n$ be any cycle   of order $n\geq 4$ with $e\in E(C)$, then  $d(C)< d(C-e)=d(P_n).$
\end{corollary}
\begin{proof}
Since there are no  true twins in any  cycle   $C_n$ of order $n\geq 4$, the result follows by Lemma  \ref{01}.
\end{proof}

%

Indeed, Lemma \ref{01} provides a transformation that strictly increases the $d(G)$-index. Subsequently, it is a natural progression to explore an alternative transformation that also results in a strict increase of the $d(G)$-index.
Let $u_q$ be a quasi-pendant vertex in graph $G$, connected to $s \geq 2$ pendant vertices, denoted by $v_1,v_2,\ldots,v_s$. Consider the subgraph $G[S]$ induced by the set $S=\{u_q\cup v_1\cup  \cdots\cup v_s\}$. We construct a new graph $G'$ by replacing $G[S]$ with the graph $K_1\ast (\lfloor \frac{s}{2} \rfloor K_{2}\cup (s\bmod 2)K_{1})$ (refer to Figure $2$ for illustration).
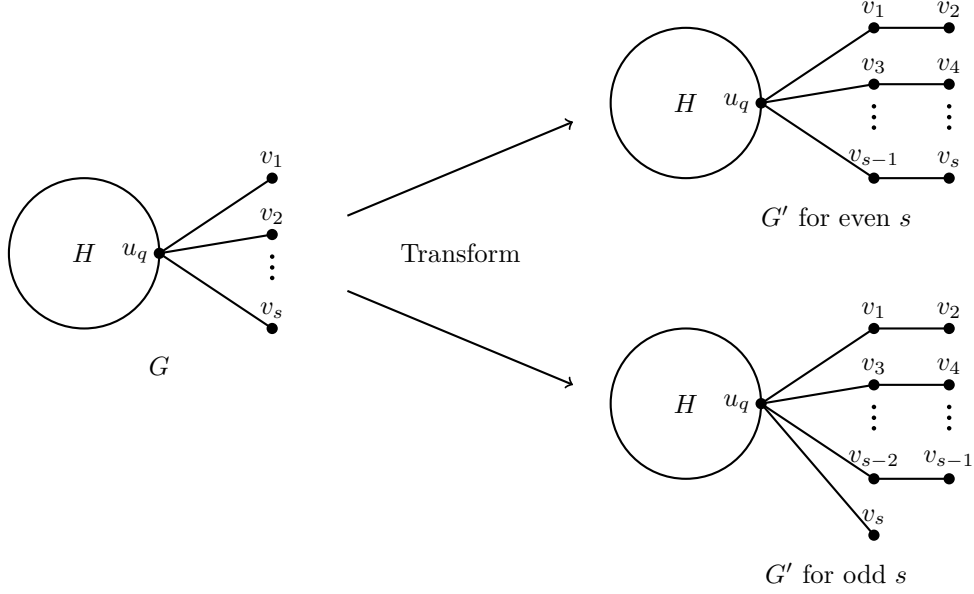
\begin{figure}[h]
		\centering
		\begin{tikzpicture}[node distance=1.5cm and 2cm, thick,
			every node/.style={font=\small},
			arrow/.style={-Latex, thick}]
			
			
			\coordinate [label=center: $H$] (A) at (-5,0);
			\coordinate [label=center: $G$] () at (-4,-1.5);
			\coordinate [label=left: $u_q$] (vp) at (-4,0);
			\coordinate [label=above: $v_1$] (v1) at (-2.5,1);
			\coordinate [label=above: $v_2$] (v2) at (-2.5,0.25);
			\coordinate [label=above: $v_s$] (vs) at (-2.5,-1);
			\coordinate [label=above: \huge{${\vdots}$}] () at (-2.5,-0.5);
			
			\node at (vp)[circle, fill, inner sep=1.5pt]{};
			\node at (v1)[circle, fill, inner sep=1.5pt]{};
			\node at (v2)[circle, fill, inner sep=1.5pt]{};
			\node at (vs)[circle, fill, inner sep=1.5pt]{};
			
			\draw (-5,0) circle (1);
			\draw (vp)--(v1);
			\draw (vp)--(v2);
			\draw (vp)--(vs);
			

			\coordinate [label=center: $H$] (B) at (3,2);
			\coordinate [label=center: $G'$ for even $s$] () at (5,0.5);
			\coordinate [label=left: $u_q$] (1vp) at (4,2);
			\coordinate [label=above: $v_1$] (1v1) at (5.5,3);
			\coordinate [label=above: $v_3$] (1v3) at (5.5,2.25);
			\coordinate [label=above: $v_{s-1}$] (1vs-1) at (5.5,1);
			\coordinate [label=above: \huge{${\vdots}$}] () at (5.5,1.5);
			\coordinate [label=above: \huge{${\vdots}$}] () at (6.5,1.5);
			\coordinate [label=above: $v_2$] (1v2) at (6.5,3);
			\coordinate [label=above: $v_4$] (1v4) at (6.5,2.25);
			\coordinate [label=above: $v_{s}$] (1vs) at (6.5,1);

			\node at (1vp)[circle, fill, inner sep=1.5pt]{};
			\node at (1v1)[circle, fill, inner sep=1.5pt]{};
			\node at (1v3)[circle, fill, inner sep=1.5pt]{};
			\node at (1vs-1)[circle, fill, inner sep=1.5pt]{};
			\node at (1v2)[circle, fill, inner sep=1.5pt]{};
			\node at (1v4)[circle, fill, inner sep=1.5pt]{};
			\node at (1vs)[circle, fill, inner sep=1.5pt]{};
			
			\draw (3,2) circle (1);
			\draw (1vp)--(1v1);
			\draw (1vp)--(1v3);
			\draw (1vp)--(1vs-1);
			\draw (1v1)--(1v2);
			\draw (1v3)--(1v4);
			\draw (1vs-1)--(1vs);
			
							
			\coordinate [label=center: $H$] (C) at (3,-2);
			\coordinate [label=center: $G'$ for odd $s$] () at (5,-4.25);
			\coordinate [label=left: $u_q$] (2vp) at (4,-2);
			\coordinate [label=above: $v_1$] (2v1) at (5.5,-1);
			\coordinate [label=above: $v_3$] (2v3) at (5.5,-1.75);
			\coordinate [label=above: $v_{s-2}$] (2vs-2) at (5.5,-3);
			\coordinate [label=above: $v_{s}$] (2vs) at (5.5,-3.75);
			\coordinate [label=above: \huge{${\vdots}$}] () at (5.5,-2.5);
			\coordinate [label=above: \huge{${\vdots}$}] () at (6.5,-2.5);
			\coordinate [label=above: $v_2$] (2v2) at (6.5,-1);
			\coordinate [label=above: $v_4$] (2v4) at (6.5,-1.75);
			\coordinate [label=above: $v_{s-1}$] (2vs-1) at (6.5,-3);

			\node at (2vp)[circle, fill, inner sep=1.5pt]{};
			\node at (2v1)[circle, fill, inner sep=1.5pt]{};
			\node at (2v3)[circle, fill, inner sep=1.5pt]{};
			\node at (2vs-2)[circle, fill, inner sep=1.5pt]{};
			\node at (2v2)[circle, fill, inner sep=1.5pt]{};
			\node at (2v4)[circle, fill, inner sep=1.5pt]{};
			\node at (2vs-1)[circle, fill, inner sep=1.5pt]{};
			\node at (2vs)[circle, fill, inner sep=1.5pt]{};
			
			\draw (3,-2) circle (1);
			\draw (2vp)--(2v1);
			\draw (2vp)--(2v3);
			\draw (2vp)--(2vs-2);
			\draw (2vp)--(2vs);
			\draw (2v1)--(2v2);
			\draw (2v3)--(2v4);
			\draw (2vs-2)--(2vs-1);
			
			\draw[-to] (-1.5,0.5)--(1.5,1.75);
			\draw[-to] (-1.5,-0.5)--(1.5,-1.75);
			\coordinate [label=center: Transform] () at (0,0);

		\end{tikzpicture}
		\caption{Graph $G$ and its transformation $G'$.}
	\end{figure}
\begin{lemma} \label{y6}
Let $G$ and $G'$  be defined as above. Then $d(G)< d(G').$
\end{lemma}
\begin{proof}
Note that $v_1,v_2,\ldots,v_s$ are pairwise false twins in the graph $G$, then either $d(G)=d(G+v_1v_2+v_3v_4+\cdots +v_{s-1}v_{s})$ if $s$ is even or $d(G)=d(G+v_1v_2+v_3v_4+\cdots +v_{s-2}v_{s-1})$ if $s$ is odd by Lemma  \ref{01}.
When  $s$ is even,
since $v_2,u_q$ are not true twins in the graph $G+v_1v_2+v_3v_4+\cdots +v_{s-1}v_{s}:=H$,
$d(G)=d(H)<d(H-v_2u_q)$. Repeatedly applying Lemma  \ref{01},  we have that $d(G)< d(G').$
When $s$ is odd,  it can be readily verified that $d(G)< d(G')$ for analogous reasons.
\end{proof}

\section{The maximum number of dissociation sets of a tree\label{tree}}
Let $\mathcal{T}_n$  denote the set of all trees of order $n$. A tree
$T$ within this set is designated as a {\it a maximal tree} if its
$d(T)$-index achieves the maximum value among all trees in $\mathcal{T}_n$. We now proceed to characterize the structure of the extremal trees  with the most maximum dissociation
sets as follows.
 \begin{theorem} \label{y7}
Let $T$ be a   tree of order $n$. Then  $d(T) \leq f(n)$
with equality iff $T\cong F_n$, where
\begin{align*}
f(n):=
\begin{cases}
2^{n-1}+(n+3)\cdot 2^{\frac{n-5}{2}}, &~ {\rm if}~  n~{\rm is}~{\rm odd};\\
2^{n-1}+(n+6)\cdot 2^{\frac{n-6}{2}}, &~ {\rm if}~  n~{\rm is}~{\rm even}.
\end{cases}
\end{align*}
and
\begin{align*}
F_n:=
\begin{cases}
K_1\ast \frac{n-1}{2}K_2, &~ {\rm if}~  n~{\rm is}~{\rm odd}; \\
P_6~{\rm or}~ K_2\ast 2K_2, &~ {\rm if}~ n=6; \\
K_2\ast \frac{n-2}{2}K_2, &~ {\rm if}~ n~{\rm is}~{\rm even}~{\rm and}~ n\neq 6. \\
\end{cases}
\end{align*}
\end{theorem}
\begin{proof}
First, note that $d(F_n)=f(n)$ by Proposition \ref{y1}.
We shall prove the theorem by induction on $n$.
The result holds trivially if $n\leq 6$ by Proposition \ref{y1} and Lemma \ref{y6}.
Assume that it is true for all $n'<n$.
Let $T$ be a  maximal tree of order $n\geq 7$  and choose a longest path $v_0v_1v_2\cdots v_l$   of $T$ starting from $v_0$.
Then Lemma \ref{y6} implies that each
quasi-pendant vertex of a maximal tree  is adjacent to exactly one pendant vertex.
So $d_T(v_1)=2$ and $l\geq 4$ otherwise $n\leq 4.$ By Proposition \ref{y1},
\begin{equation}\label{(1.2)}
\begin{aligned}
 d(T) =d(T-v_0)+d(T-v_0-v_1)+d(T-v_0-v_1-v_2).
\end{aligned}
\end{equation}

If $n$ is odd, then Eq.~(\ref{(1.2)}) yields
$$d(T)\leq f(n-1)+f(n-2)+g(n-3)=f(n),$$
where the inequality follows from the facts that both $T-v_0$ and $T-v_0-v_1$ are connected and  Theorem \ref{y4}, respectively.
Moreover, the equality holding imply that
$T-v_0-v_1-v_2$ is the disjoint union of $K_1,K_2$ by Theorem \ref{y4}. Combining the condition that each  quasi-pendant vertex of the maximal tree $T$ is adjacent to exactly one pendant vertex, we have that  $T\cong K_1\ast \frac{n-1}{2}K_2$.

If $n$ is even, we claim that $l=4$ otherwise $T-v_0-v_1-v_2$ contains an induced subtree of order $3\leq s\leq n-3.$
Further,  $d(T-v_0-v_1-v_2)\leq g(n-3-s)\cdot f(s)$  by Theorem \ref{y4} and  the induction hypothesis.
Consequently, Eq.~(\ref{(1.2)}) yields
\begin{equation*}
\begin{aligned}
d(T)& \leq f(n-1)+f(n-2)+g(n-3-s)\cdot f(s)\\
&= 2^{n-2}+(n+2)\cdot 2^{\frac{n-6}{2}}+2^{n-3}+(n+4)\cdot 2^{\frac{n-8}{2}}+2^{n-4}+(s+3)\cdot 2^{\frac{2n-s-11}{2}}\\
&<f(n),
\end{aligned}
\end{equation*}
which  contradicts the maximality of $T.$ (Since $d(F_n) =f(n)$, $d(T) \geq f(n)$ as $T$ is a maximal tree.)
It follows that $l=4$ and further $d_T(v_1)=d_T(v_3)=2$ by Lemma \ref{y6}.
It is easy to verify that $T\cong K_2\ast \frac{n-2}{2}K_2$ and $d(T)=2^{n-1}+(n+6)\cdot 2^{\frac{n-6}{2}}.$
The result follows.
\end{proof}

Given that there are no false twins in the maximal trees $f(n)$ as established in Theorem \ref{y7}, we can derive the following extremal values among all connected graphs as an application of Theorem \ref{y7}.
\begin{theorem}\label{y8}
Let $G$ be a  connected graph of order $n$. Then  $d(G) \leq f(n)$
with equality iff $G\cong F_n$, where $f(n)$ and $F_n$ are defined in  Theorem \ref{y7}.
\end{theorem}
\begin{proof}
Let $c(G)$ be the dimension of cycle spaces of $G$. If $c(G)=0,$ then the result follows by
Theorem \ref{y7}.
We claim that if $c(G)=d\geq 1,$ then $d(G)< f(n)$ and thus the result follows.
Since $G$ is  connected and $c(G)=d,$
we may choose the edges $u_1v_1,u_2v_2,\ldots,u_dv_d$ of  $G$
such that the subgraph $G-u_1v_1-u_2v_2-\cdots-u_dv_d:=T$ is a spanning tree of $G$. Applying Lemma \ref{01} and Theorem \ref{y7}, we have
\begin{equation*}
\begin{aligned}
d(G)& \leq d(G-u_1v_1)\leq \cdots \\
&\leq d(G-u_1v_1-u_2v_2-\cdots -u_{d-1}v_{d-1})\\
&\leq d(G-u_1v_1-u_2v_2-\cdots -u_{d-1}v_{d-1}-u_dv_d)=d(T)\\
&\leq f(n),
\end{aligned}
\end{equation*}
with equality iff all of  the above inequalities hold, where $d(T)= f(n)$ implies that $T\cong F_n$  by  Theorem \ref{y7}, and $d(G-u_1v_1-u_2v_2-\cdots -u_{d-1}v_{d-1})=d(T)$ implies that vertices $u_{d}$ and $v_{d}$ are true twins in the subgraph $G-u_1v_1-u_2v_2-\cdots -u_{d-1}v_{d-1}$ and further $u_{d}$ and $v_{d}$ are false  twins in $T$. However,  this contradicts the fact that there are no false  twins in $T\cong F_n$, where  $F_n$ is depicted in Theorem \ref{y7}. Then the result follows.
\end{proof}

\begin{question}\label{q1}
What is the number of dissociation sets in the connected graph of order $n$ with the second largest count?
\end{question}
In the proof of Theorem \ref{y8}, it is demonstrated that the extremal graph
$G$, which attains
the maximum $d(G)$ within the class of connected graphs, is isomorphic to the tree
$G\cong F_n$. Since there are no false twins in $F_n$, increasing a new edge between any two non-adjacent vertices in $F_n$ can make $d(F_n)$ strictly decrease.
Lemma \ref{01} establishes that the solution to Question \ref{q1} is realized within either unicyclic graphs or trees. 
We first address Question \ref{q1} in the context of unicyclic graphs, specifically by determining an upper bound on the maximum number of dissociation sets in a unicyclic graph. Subsequently, we propose a conjecture for the candidate graph with the second largest number of dissociation sets among the set of trees.

\section{The maximum number of dissociation sets of a unicyclic graph\label{unicyclic}}
In this section, we address the aforementioned extremal problems within the context of unicyclic graphs, aiming to establish bounds for the solution to Question \ref{q1}.
Let
\begin{equation}\label{eq}
\begin{aligned}
 h(n):=
\begin{cases}
2^{n-1}+(n+9)\cdot 2^{\frac{n-7}{2}}, &~ {\rm if}~  n~{\rm is}~{\rm odd};\\
2^{n-1}+(n+12)\cdot 2^{\frac{n-8}{2}}, &~ {\rm if}~  n~{\rm is}~{\rm even}~{\rm and}~ n\neq 6;\\
42, &~ {\rm if}~  n=6.
\end{cases}
\end{aligned}
\end{equation}
We start our investigation by excluding  some simple graph classes as candidates.
\begin{lemma} \label{u1}
Let $n$  be an integer  such that $n\geq 4.$ Then $d(C_n)<h(n).$
\end{lemma}
\begin{proof}
We first claim that when $n\geq 9,$ $d(P_n)<h(n).$ Further since $d(C_n)<d(P_n)$ by Corollary \ref{coro1}, the result follows if $n\geq 9$.
We prove the claim by induction on $n.$ Since $d(P_n)=d(P_{n-1})+d(P_{n-2})+d(P_{n-3})$ by Proposition  \ref{y2}, it is easy to see that
$d(P_9)=274<h(9)=292$,   $d(P_{10})=504<h(10)=556,$ and $d(P_{11})=927<h(11)=1104.$
Suppose now that the claim holds in the case that $|P|\leq n-1$, where $n\geq 12$.
Next we divide into two cases to show that $d(P_n)<h(n).$
If  $n$ is odd, then
\begin{equation*}
\begin{aligned}
~~~~~~~~~~~~~~~~~~~~~~~~~~~~~~~~&h(n-1)+h(n-2)+h(n-3)\\
& =2^{n-2}+(n+11)\cdot 2^{\frac{n-9}{2}}+2^{n-3}+(n+7)\cdot 2^{\frac{n-9}{2}}+2^{n-4}+(n+9)\cdot 2^{\frac{n-11}{2}}\\
&= (\frac{1}{2}+\frac{1}{4}+\frac{1}{8})\cdot 2^{n-1}+\left(\frac{n+11}{2}+\frac{n+7}{2}+\frac{n+9}{4}\right)\cdot 2^{\frac{n-7}{2}}\\
&=2^{n-1}+(n+9)\cdot 2^{\frac{n-7}{2}}-\left(\frac{2^{n-1}}{8}-\frac{n+9}{4}\cdot 2^{\frac{n-7}{2}}\right)\\
&<h(n),
\end{aligned}
\end{equation*}
where the last inequality follows  from  the fact that $\frac{2^{n-1}}{8}-\frac{n+9}{4}\cdot 2^{\frac{n-7}{2}}>0$ when $n\geq 9$.
Then $h(n)>h(n-1)+h(n-2)+h(n-3)\geq d(P_{n-1})+d(P_{n-2})+d(P_{n-3})=d(P_n)$.
For the case that $n$ is even, we can get $d(P_n)<h(n)$  by the similar way. Consequently, the claim follows.

When $4\leq n\leq 8,$ it is easy to check that $d(C_n)<h(n).$ The result follows.
\end{proof}

\begin{definition}
Consider the cycle $C_n$ as $v_1v_2\ldots v_nv_1$. Denote by $C_n(k_1,\ldots,k_n)$ the unicyclic graph obtained from  $C_n$ by identifying $v_i\in V(C_n)$ with a   vertex of pendant tree $T_{k_i}$ for $1\leq i\leq n$, where $k_1+\cdots +k_n$ equals to the order of $C_n(k_1,\ldots,k_n)$.
\end{definition}

Let  $\mathcal{U}_n$  be the set of unicyclic graphs of order $n$. A unicyclic graph  $G$  is called  {\it maximal}
if $d(G)$ attains the maximum   value  among $G\in \mathcal{U}_n$.

\begin{theorem} \label{y10}
Let $G$ be a  unicyclic graph  of order $n\geq 3$. Then  $d(G) \leq h(n)$
with equality iff $G\cong U_n$, where $h(n)$ is defined in Eq.~(\ref{eq})
and
\begin{align*}
	U_n:=
	\begin{cases}
		K_3\ast \left(\lfloor \frac{n-3}{2} \rfloor K_{2}\cup (n+1\bmod 2)K_{1}\right), &~ {\rm if}~n\neq 6;\\
		K_1\ast (K_{3}\cup K_{2}), &~ {\rm if}~  n=6.
	\end{cases}
\end{align*}
\end{theorem}
\begin{proof}
First, note that $d(U_n)=h(n)$ by Proposition \ref{y1}.
Denote by $r$ the girth of $G$. If $r=n,$ i.e. $G\cong C_n$, then the result follows by Lemma \ref{u1}.
So it suffices to consider $r<n,$ i.e., there exists some pendant trees attaching some vertices of the cycle $C_r.$

The result is trivial for  $n\leq 6$.
We assume henceforth that $n\geq 7$ and proceed by induction on $n$.
Consider then a graph $C_r(k_1,\ldots,k_r):=U$ as in the statement of the theorem with  $d(U)$ maximal.
It follows from Lemma \ref{y6}  that each
quasi-pendant vertex of the maximal unicyclic graph $U$ is adjacent to exactly one pendant vertex.
Denote by $\zeta_i$ the length of  a longest path $v_iu_1u_2\ldots u_{\zeta_i-1}u_{\zeta_i}$  of the pendant tree $T_{k_i}$ starting from $v_i\in V(C_r)$ for $i=1,\ldots,r.$ Set $l_i=\max\{\zeta_i, 1\leq i\leq r\}\geq 1$ as $U\neq C_n$.

{\bf Case 1.} {\it $l_i\geq 2$.}
In this case, $v_iu_1u_2\ldots u_{l_i-1}u_{l_i}$ is the longest path of the pendant tree $T_{k_i}$.
Since $u_{l_i-1}$ is a quasi-pendant vertex of $U$,
$d_U(u_{l_i-1})=2$ and then $U-u_{l_i}-u_{l_i-1}$ is also a  unicyclic graph.
By Proposition \ref{y1},
\begin{equation}\label{(2.2)}
\begin{aligned}
 d(U) =d(U-u_{l_i})+d(U-u_{l_i}-u_{l_i-1})+d(U-u_{l_i}-u_{l_i-1}-u_{l_i-2}).
\end{aligned}
\end{equation}

If $n$ is odd, then by induction Eq.~(\ref{(2.2)}) yields
$$d(U)\leq h(n-1)+h(n-2)+g(n-3)=h(n).$$
Moreover, the equality holding imply that
$d(U-u_{l_i}-u_{l_i-1}-u_{l_i-2})=g(n-3)$ (if $l_i=2$, $u_{l_i-2}=v_i$), which shows that $U-u_{l_i}-u_{l_i-1}-u_{l_i-2}$
is the disjoint union of $K_1$ and $K_2.$
Combining the facts that $n$ is odd and each
quasi-pendant vertex of $U$  is adjacent to exactly one pendant vertex,
we get $U\cong K_3\ast \frac{n-3}{2}K_2.$
The result follows in this case.

If $n$ is even, we first claim that  $l_i\leq 2$. Otherwise, $U-u_{l_i}-u_{l_i-1}-u_{l_i-2}$ contains a
unicyclic subgraph of order $3\leq s\leq n-3$. And further  $d(U-u_{l_i}-u_{l_i-1}-u_{l_i-2})\leq 2^{n-3-s}\cdot h(s)\leq 2^{n-6}\cdot h(3)=7\cdot 2^{n-6}$.
Now when $n\geq 8$, Eq.~(\ref{(2.2)}) yields
\begin{equation}\label{(234)}
\begin{aligned}
d(U)& \leq h(n-1)+h(n-2)+ 7\cdot 2^{n-6}\\
&= 2^{n-2}+(n+8)\cdot 2^{\frac{n-8}{2}}+ 2^{n-3}+(n+10)\cdot 2^{\frac{n-10}{2}}+7\cdot 2^{n-6}\\
&< f(n)=2^{n-1}+(n+12)\cdot 2^{\frac{n-8}{2}},
\end{aligned}
\end{equation}
which  contradicts the maximality of $U.$ (Since $d(U_n) =h(n)$, $d(U) \geq h(n)$ as $U$ is   maximal.)
Now we have $l_i= 2,$ i.e. $v_iu_1u_2$ is the longest path of the pendant tree $T_{k_i}$ starting from $v_i\in V(C_r)$.
If $r=3$ and $U\cong K_3\ast (K_{1}\cup \frac{n-4}{2}K_2)$, then   $d(U) = h(n)$.
If $U\neq K_3\ast (K_{1}\cup \frac{n-4}{2}K_2)$, no matter what $r$ is, there exists a subtree of order
$3\leq s\leq n-3$ in the subgraph $U-u_{2}-u_{1}-v_i$. It follows from Theorem \ref{y7} that
$d(U-u_{2}-u_{1}-v_i)\leq 2^{n-3-s}\cdot f(s)\leq 2^{n-6}\cdot f(3)=7\cdot 2^{n-6}$, where $f(s)$ is shown in Theorem \ref{y7}. This situation can be excluded the same as Inequality~(\ref{(234)}).

{\bf Case 2.} {\it $l_i=1$.}
We will claim that the maximal unicyclic graph  $U$ does not exist in this case by proving that $d(U) <h(n).$
Recall that each
quasi-pendant vertex of $U$  is adjacent to exactly one pendant vertex, thus every vertex $v\in V(C_r)$ is
adjacent to at most  one pendant vertex and $|T_{k_j}|\leq 2$ for any $1\leq j\leq r$.
For $r\in \{3,4\}$, it is easy to check that $d(U) <h(n).$ Now consider that $r\geq 5$.
Denote by $v_iu_1$ the unique longest path of the pendant tree $T_{k_i}$.
Observe that $U-u_1-v_i$ is a subtree of order $n-2$, and thus $d(U-u_1-v_i)\leq f(n-2)$ by Theorem \ref{y7}.
Applying Proposition \ref{y1} to the vertex $u_1$, we have
\begin{equation*}\label{(23)}
\begin{aligned}
d(U)& \leq h(n-1)+f(n-2)+ 4\cdot f(n-6)\\
&\leq 2^{n-2}+(n+8)\cdot 2^{\frac{n-8}{2}}+ 2^{n-3}+(n+1)\cdot 2^{\frac{n-7}{2}}+4\cdot (2^{n-7}+(n-3)\cdot 2^{\frac{n-11}{2}})\\
&=2^{n-2}+ 2^{n-3}+(n+8)\cdot 2^{\frac{n-8}{2}}+2^{n-5}+(n-1)\cdot 2^{\frac{n-5}{2}}\\
&<  2^{n-1}+(n+12)\cdot 2^{\frac{n-8}{2}}\leq h(n),
\end{aligned}
\end{equation*}
where the last two inequalities follow from the facts that  $2^{n-5}+(n-1)\cdot 2^{\frac{n-5}{2}}<2^{n-3}$ if $n\geq 8,$ and in the expression of $h(n)$: $2^{n-1}+(n+9)\cdot 2^{\frac{n-7}{2}}>2^{n-1}+(n+12)\cdot 2^{\frac{n-8}{2}}$.
The result follows.
\end{proof}

Note that there is exactly one pair of true twins, $u$ and $v$, in the extremal graph $U_n$ of Theorem \ref{y10}. Then, by Lemma \ref{01}, we have $d(U_n) = d(U_n - uv)$, where $U_n - uv \cong K_1 \ast (2K_1 \cup \frac{n-3}{2}K_2)$ if $n$ is odd, and $U_n - uv \cong K_1 \ast (3K_1 \cup \frac{n-4}{2}K_2)$ if $n$ is even. We conclude this article with an open question.
\begin{question}\label{q1}
If $n \geq 10$, is $h(n)$, as defined in Eq.~(\ref{eq}), the second largest number of dissociation sets among all connected graphs of order $n$?
Does the extremal graph belong to $\{U_n, K_1\ast (2K_{1}\cup \frac{n-3}{2}K_{2})(with ~n~ odd), K_1\ast (3K_{1}\cup \frac{n-4}{2}K_{2})(with ~n~ even)\}$?
\end{question}

\end{document}